\documentclass[letterpaper, 10 pt, conference]{ieeeconf}

\IEEEoverridecommandlockouts
 \usepackage[usenames,dvipsnames,svgnames,table]{xcolor}

 \usepackage{amsmath, amssymb, bbm, xspace}
 \usepackage{amsfonts,mathrsfs}
 \usepackage{mathtools}
 \usepackage{color}
 \usepackage{graphicx}
 \usepackage{epsfig}
 \usepackage{algorithm}
 \usepackage{algorithmicx}
 \usepackage[acronym]{glossaries}
 \usepackage{subfigure}
 \usepackage{cancel}
 \usepackage{empheq}
 \usepackage{placeins}
 \let\labelindent\relax
 \usepackage{enumitem}
 \usepackage{textgreek}
 \usepackage{mathbbol}
 \usepackage{stackengine}
\usepackage{placeins}
\usepackage{dblfloatfix}
\usepackage{cancel}

 \def\QEDhereeqn{\eqno\let\eqno\relax\let\leqno\relax\let\veqno\relax\hbox{\QED}}
 \def\QEDopenhereeqn{\eqno\let\eqno\relax\let\leqno\relax\let\veqno\relax\hbox{\QEDopen}}
 \newcommand{\bs}{\boldsymbol}
 \newcommand{\mc}{\mathcal}
 \renewcommand{\emph}{\textit}
 
 \newcommand{\0}{\bs 0}
 \def\1{{\bs 1}}
 
\def\min{\mathop{\rm min}}
 \newcommand{\col}{\mathrm{col}}
 \def\diag{\mathop{\hbox{\rm diag}}}

 \def\R{\mathbb{R}}
 \def\N{\mathbb{N}}

 \DeclareSymbolFontAlphabet{\mathbbm}{bbold}
 \DeclareSymbolFontAlphabet{\mathbb}{AMSb}%
 
 \stackMath
	 \newcommand\tsup[2][2]{%
	 	\def\useanchorwidth{T}%
	 	\ifnum#1>1%
	 	\stackon[-.5pt]{\tsup[\numexpr#1-1\relax]{#2}}{\scriptscriptstyle\sim}%
	 	\else%
	 	\stackon[.5pt]{#2}{\scriptscriptstyle\sim}%
	 	\fi%
	 }

\makeglossaries
\newacronym{KKT}{KKT}{Karush--Kuhn--Tucker}
\newacronym{ADMM}{ADMM}{alternating direction method of multipliers}
\newacronym{OPF}{OPF}{optimal power flow}
\newacronym{OPFP}{OPFP}{optimal power flow problem}
\newacronym{NUM}{NUM}{network utility maximization}
\newacronym{LMI}{LMI}{linear matrix inequality}
\newacronym{BMI}{BMI}{bilinear matrix inequality}
\newacronym{LM}{LM}{Lyapunov-Metzler}
\newacronym{SDP}{SDP}{semidefinite programming}
\newacronym{LTI}{LTI}{linear time invariant}
\newacronym{MJLS}{MJLS}{Markov jump linear system}
\newacronym{PID}{PID}{proportional-integral-derivative}
\newacronym{PPA}{PPA}{proximal-point algorithm}
\newacronym{PPPA}{PPPA}{preconditioned proximal-point algorithm}
\newacronym{PPP}{PPP}{preconditioned proximal-point}
\newacronym{NE}{NE}{Nash equilibrium}
\newacronym{GNE}{GNE}{generalized Nash equilibrium}
\newacronym{v-GNE}{v-GNE}{variational GNE}
\newacronym{ISS}{ISS}{input-to-state stability}

 \makeatletter \let\cl@part\relax \makeatother
 \usepackage{nohyperref}
 \usepackage[capitalize]{cleveref}
  \let\k\relax
 \def\k{{k \in \N}}  
 \usepackage{letltxmacro}

 \crefname{thm}{Theorem}{Theorems}
 \crefname{lem}{Lemma}{Lemmas}
 \crefname{cor}{Corollary}{Corollaries}
 \crefname{rem}{Remark}{Remarks}
 \crefname{alg}{Algorithm}{Algorithm}
 \crefname{figure}{Figure}{Figures}
 \crefname{assumption}{Assumption}{Assumptions}
 \crefname{corollary}{Corollary}{Corollaries}
\crefname{proposition}{Proposition}{Propositions}
\crefname{problem}{Problem}{Problems}
\crefname{example}{Example}{Examples}
\crefname{definition}{Definition}{Definition}
\crefname{condition}{C}{C}

 \crefname{thmlisti}{Theorem}{Theorem}
 \crefname{lemlisti}{Lemma}{Lemma}
 \crefname{asmlisti}{Assumption}{Assumptions}

 \newlist{thmlist}{enumerate}{1}
 \setlist[thmlist]{label=(\roman{thmlisti}), ref=\thethm(\roman{thmlisti}),noitemsep}
 \newlist{lemlist}{enumerate}{1}
 \setlist[lemlist]{label=(\roman{lemlisti}), ref=\thelem(\roman{lemlisti}),noitemsep}
 \newlist{asmlist}{enumerate}{1}
 \setlist[asmlist]{label=(\roman{asmlisti}), ref=\theassumption(\roman{asmlisti}),noitemsep,nosep,leftmargin=*} 

 \newtheorem{lemma}{Lemma}
 \newtheorem{theorem}{Theorem}
 \newtheorem{remark}{Remark}
 \newtheorem{assumption}{Assumption}

\newtheorem{definition}{Definition}

\usepackage{etoolbox}
\patchcmd{\smallmatrix}{\thickspace}{\kern.5em}{}{}
\newenvironment{smallbmatrix}
  { \left[\begin{smallmatrix}}
  {\end{smallmatrix}\right]}

\usepackage{caption}
\usepackage{booktabs}
\usepackage{multirow}

\DeclareSymbolFont{myletters}{OML}{ztmcm}{m}{it}
\DeclareMathSymbol{\uplambda}{\mathord}{myletters}{"15}

\begin{document}

{\title{
Linear convergence in time-varying generalized Nash equilibrium problems
}

\author{Mattia Bianchi, Emilio Benenati, Sergio Grammatico
\thanks{
Emilio Benenati and Sergio Grammatico are with the the Delft Center for Systems and Control, TU Delft, The Netherlands. Mattia Bianchi is with the Automatic Control Laboratory, ETH Z\"{u}rich, Switzerland. 
E-mail addresses: \texttt{\{e.benenati, s.grammatico\}@tudelft.nl, mbianch@ethz.ch}.
 This work was  supported by the ERC under research project COSMOS (802348).
}
}

\maketitle
\thispagestyle{empty}
\pagestyle{empty}

\begin{abstract} 
We study generalized games with full row rank equality constraints and we provide a strikingly simple proof of strong monotonicity of the associated KKT operator. This allows us to show linear convergence to a variational equilibrium of the resulting primal-dual pseudo-gradient dynamics. Then, we propose a fully-distributed algorithm with linear convergence guarantee for aggregative games under partial-decision information. 
Based on these results, we establish stability properties for online GNE seeking in games with time-varying cost functions and constraints. Finally, we illustrate our findings numerically on an economic dispatch problem for peer-to-peer energy markets.
\end{abstract}


\section{Introduction}\label{sec:introduction}

\Gls{GNE} problems arise in many multi-agent applications, where the agents are coupled  not only because of their conflicting objectives, but also via shared constraints --operational limits of the  system, that the agents should respect. Among others, \gls{GNE} seeking is used in  energy markets \cite{BelgioiosoAnanduta_TSG2022}, radio communication \cite{Facchinei_Kanzow_2010} and formation control \cite{Simaan_AUT_2019} problems. 

The networked structure of these applications naturally calls for distributed solution methods. In fact, part of the recent literature focuses on semi-decentralized \gls{GNE} seeking algorithms \cite{Belgioioso_Grammatico_aggregative_TAC_2021,Paccagnan_NashAndWardrop_TAC2019,Benenati_CDC2022}, where the agents update their decision locally, with the help of  a  coordinator that gathers and broadcasts information over the systems (a setup also named full-information scenario). Other works \cite{Gadjov_Pavel_aggregative_2019,Pavel_GNE_TAC_2020,Bianchi_GNEPPP_AUT_2022,Deng_Nian_Aggregative_TNNLS_2019} deal with applications where the agents can only rely on fully-distributed peer-to-peer communication and local data. In this so-called partial-decision information scenario, the agents compensate for the lack of global knowledge by estimating the unknown quantities and by embedding consensus dynamics in their local decision processes.

In both scenarios, to cope with the presence of coupling constraints and to distribute the computation among the agents, one should resort to Lagrangian reformulations.  In fact, all the references  above leverage  primal-dual pseudo-gradient algorithms, aimed at solving the \gls{KKT} optimality  conditions of the \gls{GNE} problem.

In general, primal-dual  algorithms fails to achieve linear convergence, even for the class of strongly-monotone generalized games \cite{Belgioioso_Grammatico_aggregative_TAC_2021}. Importantly, together with linear convergence,  some crucial  \gls{ISS} properties of  pseudo-gradient iterations are also not guaranteed. This lack of robustness is a critical issue for methods in the partial-decision information scenario, where convergence should be ensured despite the estimation error. To overcome this complication,  vanishing step sizes can be used to drive the error to zero \cite{Belgioioso_Nedic_Grammatico_2020}, at the price of slow convergence. Alternatively, several fixed-step algorithms for \gls{GNE} seeking were derived based on operator-theoretic methods and on the use of preconditioning \cite{Pavel_GNE_TAC_2020,Bianchi_GNEPPP_AUT_2022,Huang_Hu_DR_CDC_2021,Gadjov_Pavel_aggregative_2019}. Unfortunately, this approach comes with important limitations, such as extending the analysis to time-varying setups. 

For instance, 
there is no available fixed-step fully-distributed method to solve \gls{GNE} problems when the agents can only exchange information over switching communication networks (while methods are available for games without coupling constraints \cite{Bianchi_TV_LCSS_2021,Gadjov_adversaries_CDC_2022}). 
Furthermore, in many decision processes with real-world applications, the cost functions of the agents and the  system constraints can vary over time \cite{DallAnese_Optimization_SPM_2020}, for instance in cognitive radio networks and demand response in smart grids \cite{Basar_Confluence_CSM_2022}.
In such domains, linearly convergent algorithms become particularly desirable, as the solver needs to quickly update its solution in response to changes in the environment.  Despite its practical relevance, there are very few works that study the online \gls{GNE} problem. The paper \cite{su_Online_AIS_2021} proposes a regularized algorithm,  which only achieves inexact convergence, and which is not fully-distributed. Instead, the authors of \cite{Lu_Online_TAC_2021}  develop an algorithm for the  partial-decision information scenario, but that 
achieves sublinear regret only when the solution is asymptotically constant and for diminishing step sizes. 

\emph{Contribution:} In this paper we study  generalized
games with full row rank coupling equality constraints --as those arising in resource allocation and transportation problems  \cite{Stein_transportationgames_EJOR_2018}, where demand-matching \cite{Deng_Nian_Aggregative_TNNLS_2019} and flow \cite{BelgioiosoAnanduta_TSG2022} constraints are ubiquitous.  For the first time, we show that, in this setup, linear convergence to a \gls{GNE} can be achieved via primal-dual dynamics, both for the full- and partial-decision information scenario. Thanks to this result, we can also adapt the dynamics to online equilibrium seeking in time-varying games. Here we focus on the prominent class of aggregative games \cite{Belgioioso_Grammatico_aggregative_TAC_2021}, for its desirable scalability properties, but the analysis carries over to generally-coupled costs.  We summarize the novelties of our work  as follows:

\begin{enumerate}
    \item We provide a simple, constructive proof of the strong monotonicity of the \gls{KKT} operator in games with full-row rank equality coupling constraints. As a consequence, we show linear convergence to a \gls{GNE} of the
    pseudo-gradient ascent-descent method (Section~\ref{sec:Linearconvergence});
    \item We design a linearly convergent algorithm for \gls{GNE} seeking in partial-decision information, via a tracking technique \cite{Belgioioso_Nedic_Grammatico_2020} that avoids the need for slack variables. Our proof is based on a change of coordinates and a small gain argument: due to its generality, the argument also applies to the case of (Q-connected) time-varying communication  graphs (Section~\ref{sec:partial_info});
    \item We exploit our linear convergence results to study the tracking properties of the proposed methods with respect to the solution of a game with time-varying costs and constraints. In particular, for the fully-distributed algorithm, we show that the extra error induced in the dynamic tracking procedure does not jeopardize stability (Section~\ref{sec:timevarying}). 
    \end{enumerate}

\medskip
\emph{Notation:}
$\0_n$ ($\1_n$) denotes the vector of dimension $n$ with all elements equal to $0$ ($1$); $I_n$   the identity matrix of dimension $n$; the subscripts are omitted when there is no ambiguity. If $A$ is symmetric,   $\uplambda_{\textnormal{min}}(A)=\uplambda_1(A)\leq\dots\leq\uplambda_n(A)=:\uplambda_{\textnormal{max}}(A)$ denote its eigenvalues.
$\otimes$ denotes the Kronecker product.
$\diag(A_1,\dots,A_N)$ denotes the block diagonal matrix with $A_1,\dots,A_N$ on its diagonal; $\col\left(x_1,\ldots,x_N\right) = [ x_1^\top \ldots  x_N^\top ]^\top$. For a positive definite symmetric matrix $P\succ 0$, $\langle x, y \rangle_P = x^\top P y$ denotes the $P$-weighted inner product, $\|x\|_P$ the corresponding norm; we omit the subscript if $P=I$. An operator $\mathcal{F} : \R^n \rightarrow \R^n$ is
($\mu$-strongly) monotone in $\mc{H}_P$ if, for any $x,y\in\R^n$, $\langle \mc{F}(x)-\mc{F}(y),x-y \rangle_P \geq 0 (\geq \mu \|x-y\|_P^2)$; is contractive in $\mc{H}_P$ if it is Lipschitz with constant smaller than 1, i.e., for some $\ell<1$ and for any $x,y\in\R^n$, $\|\mc{F}(x)-\mc{F}(y)\|_P \leq \ell \| x-y\|_P$; we omit the indication ``in $\mc{H}_P$'' if $P = I$. If $\mc{F}$ is differentiable, $D\mc{F}$ denotes its Jacobian.

\section{Mathematical setup}

We consider a set of  agents, $ \mc I\coloneqq \{ 1,\ldots,N \}$, where each agent $i\in \mc{I}$ shall choose its decision variable (i.e., strategy) $x_i \in \R^{n_i}$.  Let $x \coloneqq  \col( (x_i)_{i \in \mc I})  \in \R^n $ denote the stacked vector of all the agents' decisions, with $\textstyle n\coloneqq \sum_{i=1}^N n_i$.
The goal of each agent $i \in \mc I$ is to minimize its objective function $J_i(x_i,x_{-i})$, which depends on both the local variable $x_i$ and on the decision variables of the other agents $x_{-i}\coloneqq  \col( (x_j)_{j\in \mc I\backslash \{ i \} })$.
Furthermore, the feasible decisions of each agent depends   on the action of the other agents via affine equality coupling constraints. Specifically, the feasible set is $\mc{X} \coloneqq \left\{x \in \R^{n} \mid Ax =  b \right\}$,
	where $A\coloneqq \left[A_{1}, \ldots, A_{N}\right]$ and $b\coloneqq \sum_{i=1}^{N} b_{i}$,  $A_{i} \in \R^{m \times n_{i}}$ and $b_{i} \in \R^{m}$ being locally available information. The game is then represented by the inter-dependent optimization problems:
	\begin{align} \label{eq:game}
	(\forall i \in \mathcal{I})
	\  \min_{ y_i \in \R^{ {n_{i}}}}   \; J_i(y_i,x_{-i}) \quad
	\text{s.t.}  \  (y_i,x_{-i}) \in \mathcal X.
	\end{align}
	The technical problem we consider here is the computation of a \gls{GNE}, namely a set of decisions that simultaneously solve all the optimization problems in  \eqref{eq:game}.
	\begin{definition}
		A collective strategy $x^{\star}=\operatorname{col}\left((x_{i}^{\star}\right)_{i \in \mathcal{I}})$ is a generalized Nash equilibrium if, for all $i \in \mc{I}$,
	$
		J_{i}\left(x_i^{\star}, x_{-i}^{\star}\right)\leq \inf \{J_{i}\left(y_{i}, x_{-i}^{\star}\right) \mid (y_i,x_{-i}^{\star})\in\mc{X} \}.
		$ \hfill $\square$
	\end{definition}
	Next, we postulate some standard regularity and convexity
	assumptions for the constraint sets and cost functions. 
	\begin{assumption}[Convexity]\label{asm:Convexity}
		In \eqref{eq:game}, $\mc{X}$ is non-empty. For each $i\in \mathcal{I}$,  $J_{i}$ is continuous and $J_{i}\left(\cdot, x_{-i}\right)$ is convex and continuously differentiable for every $x_{-i}$.
		{\hfill $\square$} \end{assumption}
	As common in the literature \cite{Parise_Almost_TCNS_2020}, \cite{Belgioioso_Grammatico_aggregative_TAC_2021}, among all the \glspl{GNE}, we focus on the subclass of \glspl{v-GNE} \cite[Def.~3.11]{Facchinei_Kanzow_2010}, which are more economically justifiable, as well as computationally tractable  \cite{Kulkarni_Shanbag_2012}.
    Under Assumption~\ref{asm:Convexity} and defining the \emph{pseudo-gradient} mapping of the game
	\begin{align}
	\label{eq:pseudo-gradient}
	F(x)\coloneqq \operatorname{col}\left( (\nabla _{\!x_i} J_i(x_i,x_{-i}))_{i\in\mathcal{I}}\right),
	\end{align}
	$x^{\star}$ is a \gls{v-GNE} of the game in \eqref{eq:game} if and only if there exists a dual variable $\lambda^{\star}\in \R^m $ such that the following \gls{KKT} conditions are satisfied \cite[Th.~4.8]{Facchinei_Kanzow_2010}:
	\begin{align} \label{eq:KKT}
	\begin{aligned}{\0_{n}} & \in F\left(x^{\star}\right)+A^\top \lambda^{\star}
	\\
	{\0_{m}}                 
    & \in-\left(A x^{\star}-b\right).\end{aligned}
	\end{align}
Let us restrict our attention to strongly monotone games. 
	\begin{assumption}[Strong monotonicity]\label{asm:StrMon}
		The game mapping $F$ in \eqref{eq:pseudo-gradient}  is $\mu_F$-strongly monotone and $\ell_F$-Lipschitz continuous, for some $\mu_F$, $\ell_F>0$.
		\hfill $\square$
	\end{assumption}
The strong monotonicity of $F$ is sufficient to ensure existence and uniqueness of a \gls{v-GNE} \cite[Th. 2.3.3]{Facchinei_Pang_2007}; it implies strong convexity of each function $J_i(\cdot,x_{-i})$ for any fixed $x_{-i}$, but not joint convexity of the function $J_i(\cdot)$. We emphasize that strong monotonicity is a very common condition for algorithms with linear convergence. In addition, we make the following assumption. 

\begin{assumption}[Full rank constraints]\label{asm:constraintrank}
    $A$ is full row rank. $AA^\top  \geq {\mu_A} I$, $\|A\| \leq \ell_A$ for some scalars $\mu_A, \ell_A>0$. \hfill $\square$ 
\end{assumption}
Assumption~\ref{asm:constraintrank} postulates that there are no redundant constraints (or equivalently that redundant constraints are removed). This condition is well known in duality theory and optimization, as it ensures the uniqueness of dual solutions (as it can be  inferred by \eqref{eq:KKT}).

\subsection{Aggregative games}

For ease of presentation, we will specialize our results to the prominent class of aggregative games\footnote{Similar results would hold for generally-coupled games. We note that, in principle, one could set $\sigma(x)=x$ by opportunely choosing the functions $\phi_i$'s in \eqref{eq:def_sigma}}, which arises in a variety of engineering applications, e.g.,  network congestion control and demand-side management \cite{Grammatico_2017}. In particular, we assume that the cost function $J_i$ of each agent $i$ depends only on the  local decision $x_i$ and on an aggregation value
\begin{align}\label{eq:def_sigma}
    \sigma(x) = \frac{1}{N} \sum \phi_i(x_i),
\end{align}
where $\phi_i:\R^{n_i} \rightarrow \R^{q}$ is a local function of agent $i$. In short,  overloading the function $J_i$ with some abuse of notation, we also write 
\begin{align}
    J_i(x_i,x_{-i}) = J_i(x_i,\sigma(x)).
\end{align}
\begin{assumption}\label{asm:aggregation}
For each $i \in \mc{I}$, the function $\phi_i$ in \eqref{eq:def_sigma} is continuously differentiable and $\ell_\sigma$-Lipschitz continuous.
\hfill $\square$
\end{assumption}

\section{Linear convergence in generalized games}
\label{sec:Linearconvergence}

We start by showing that the \gls{KKT} operator is strongly monotone, in a suitable norm, under Assumptions~\ref{asm:StrMon} and \ref{asm:constraintrank}. 
\begin{lemma} \label{lem:KKT_strong_mon}
Let $\omega \coloneqq \col(x,\lambda)$. The operator 
\begin{align}
\omega \mapsto \mc{A}(\omega) =
\begin{bmatrix}
 F(x) + A^\top \lambda
 \\
 -Ax+b  
\end{bmatrix}
\end{align}
is $\mu_{\mc{A}}$-strongly monotone in $\mc{H}_P$, for some $P \succ 0$ and $\mu_{\mc{A}}>0$. \hfill $\square$
\end{lemma}
\begin{proof}
For some   $0 < \nu < \ell_A$, let 
\begin{align} \label{eq:def_P}
P = 
\begin{bmatrix}
 I & \nu A^\top 
 \\
\nu A &  I
\end{bmatrix} \succ 0. 
\end{align}
For any $x,x^\star \in \R^n$ and $\lambda,\lambda^\star \in \R^m$, we have
\begin{align*}
    & \hphantom{{}={}}
    \langle  \mc{A}(\omega)-\mc{A}(\omega^\star), \omega- \omega ^\star\rangle_ P 
    \\
     & = 
    \langle F(x)-F(x^\star) , x- x ^\star \rangle  
    \\
     & \hphantom{{}={}} 
      +\nu\langle F(x)-F(x^\star) , A^\top (\lambda-\lambda^\star) \rangle 
     \\
       & \hphantom{{}={}} 
     + \cancel{ \langle A^\top (\lambda-\lambda^\star), x-x^\star \rangle}
     \\
     & \hphantom{{}={}} 
     + \nu \langle A^\top (\lambda-\lambda^\star), A^\top(\lambda-\lambda^\star) \rangle
     \\
    & \hphantom{{}={}} 
    -\nu \langle  A(x-x^\star), A(x-x^\star) \rangle 
    \\
    & \hphantom{{}={}}
    - \cancel{\langle  A(x-x^\star), \lambda-\lambda^\star \rangle}
    \\
    & \geq 
    (\mu_F - \nu\ell_A^2) \|x-x^\star \|^2
    \\
    & \hphantom{{}={}}
    -\nu \ell_F \ell_A \|x-x^\star\| \|\lambda-\lambda^\star \| 
    \\ 
    & \hphantom{{}={}}
    +\nu \mu_A  \| \lambda-\lambda^\star\|^2
\end{align*}
which is positive definite for $0< \nu < \frac{4\mu_F \mu_A }{\ell_F^2\ell_A^2 +{4}\mu_A \ell_A^2} $. The conclusion follows by equivalence of norms.
\end{proof}
Based on Lemma \ref{lem:KKT_strong_mon} we can prove linear convergence of  classic primal-dual iterations.  
\begin{theorem}[GNE seeking in full-decision information]
\label{th:1} Let Assumptions~\ref{asm:Convexity}-\ref{asm:constraintrank} hold and $\ell_{\mc{A}} := (\ell_F+\ell_A) {\sqrt{\frac{\uplambda_{\max}(P)}{\uplambda_{\min} (P)}}}$. For any $0 < \alpha < \frac{2 \mu_{\mc{A}}}{\ell_{\mc{A}}^2}$ and any initial condition $(x^0,\lambda^0)$,  the iteration
\begin{subequations}\label{eq:algorithm_full}
\begin{align}\label{eq:forward}
    x^{k+1} &= x^k- \alpha\left(F(x^k)+A^\top \lambda^k \right)
    \\
    \lambda^{k+1} &= \lambda^k + \alpha \left (Ax^k-b \right)
    \end{align}
\end{subequations}
converges linearly to the unique solution $\omega^\star = (x^\star,\lambda^\star)$ of the \gls{KKT} conditions in \eqref{eq:KKT}: for all $k\in \mathbb{N}$
\begin{align}
\begin{split}\label{eq:contraction}
    \| \omega^{k+1} - \omega^\star \|^2_P \leq \rho \| \omega^{k} - \omega^\star \|^2_P,
\end{split}
\end{align}
where $\rho = 1 - {2}\alpha \mu_{\mc{A}}+\alpha^2 \ell_{\mc{A}}^2 <1$. 
\hfill $\square$
\end{theorem}
\begin{proof}
    By Lemma~\ref{lem:KKT_strong_mon}, 
$    \| \omega - \alpha\mc{A}(\omega) - (\omega ^\star- \alpha\mc{A}(\omega^\star))\|_P =
         \| \omega- \omega^\star \|^2_P -2 {\alpha} \langle \mc{A}(\omega)-\mc{A}(\omega^\star), \omega-\omega^\star \rangle_P + \alpha^2 \|\mc{A}(\omega)-\mc{A}(\omega^\star)\|_P ^2 
         \leq \rho  \|\omega- \omega^\star \|^2_P$,
and the conclusion follows because \eqref{eq:algorithm_full} can be rewritten as $\omega^{k+1} = \omega^{k}-\alpha \mc{A}(\omega^{k})$.
\end{proof}

For the case of aggregative games, \eqref{eq:forward} can be implemented in a semi-decentralized way, as in Algorithm~\ref{algo:0}.

\begin{algorithm}[t] \label{algo:0}\caption{Semi-decentralized  GNE seeking} \label{algo:0}

\emph{Iterate to convergence:} for all $k\in \mathbb{N}$,
\begin{itemize}
    \item  Each $i \in \mc{I}$: receive $\lambda^k$, $\sigma^k$ from coordinator; update
    \begin{align*}
        x_i^{k+1} & = x_i^k - \alpha \left ( \nabla_{x_i} J_i(x_i^k, {\sigma^k}) + A_i^\top {\lambda}^k \right) 
    \end{align*}
    \item Coordinator: receive $\{\phi_i({x_i}^{k+1}),A_i x_i^{k} - b_i \}_{i \in \mc{I}}$; update
      \begin{align*}
        \lambda^{k+1} & = \lambda^k + \alpha  \textstyle\sum_{i\in\mc{I}} (A_ix_i^k -b_i),  \\
        \sigma^{k+1} &= \textstyle \frac{1}{N}\sum _{i\in\mc{I}} \phi_i(x_i^{k+1})
    \end{align*}
\end{itemize}
\end{algorithm}

\section{Partial-decision information}
\label{sec:partial_info}

In this section, we  consider aggregative games in the so called-partial-decision information scenario, where there is no central coordinator, and the agents can only exchange information via peer-to-peer communication over a communication graph $\mc{G} =(\mc{I},\mc{E})$, with weight matrix $W \in \R^{N\times N}$, and $w_{i,j} \coloneqq [W]_{i,j}>0$ if and only if $(j,i)$ belongs to the set of edges $\mc{E} \subseteq \mc{I} \times \mc{I}$.
\begin{assumption}[Communication]\label{asm:graph}
The graph $\mc{G}$ is strongly connected. The weight matrix $W$ satisfies:
\begin{itemize}
    \item \emph{Double stochasticity}: $\1^\top W = \1^\top, W\1 = W \1$;
    \item \emph{Self-loops}: $w_{i,i} >0 $ for all $i\in \mc{I}$.
    \end{itemize}
We denote $\theta \coloneqq \|W- \frac{1}{N} \1 \1^\top \| <1$. 
\hfill $\square$
\end{assumption}
To remedy the lack of global knowledge, we let each agent $i\in \mc{I}$ keep:
\begin{itemize}
    \item $\bs{\sigma}_i \in \R^{q}$: estimate of the aggregation $\sigma(x)$;
    \item $\bs\lambda_i \in \R^m$: estimate of the dual variable $\lambda$;
    \item $\bs{r}_i\in \R^m$: estimate of the residual $Ax-b$;
    \item $z_i \in \R^m$: additional dual variable.
\end{itemize}
Our proposed dynamics are illustrated in Algorithm~\ref{algo:1}, where
\begin{align}
\bs{F}_i(x_i,\bs{\sigma}_i) \coloneqq 
\nabla_{x_i}J_i(x_i,\bs{\sigma}_i)+ \frac{1}{N} [D \phi_i(x_i)]^\top \nabla_{\bs{\sigma}_i} J_i(x_i,\bs\sigma_i).
\end{align} 
\begin{algorithm}[t] \caption{Fully-distributed GNE seeking} \label{algo:1}
\emph{Initialization}: choose $\alpha>0$ as in Theorem~\ref{th:2}; for all $i\in \mc{I}$, set $x_i^0 \in \R^{n_i}$, $\bs\sigma_i^0 = \phi_i(x_i^0)$, $ z_i^0 \in \R^m$, $\bs{\lambda}_i ^0 = z_i^0$, $\bs r_i^0 = A_ix_i^0 - b_i$.
\newline
\emph{Iterate to convergence}: for all $k \in \mathbb{N}$, for all $i\in \mc{I},$
\begin{itemize}
    \item  Local variables update: 
    \begin{align*}
        x_i^{k+1} & = x_i^k - \alpha \left ( \bs{F}_i(x_i, \bs{\sigma}_i^k) + A_i^\top \bs{\lambda}_i^k \right) 
        \\
        z_i^{k+1} & = z_i^k + \alpha N \bs r_i ^k
    \end{align*}
    \item Tracking: Agent $i$ exchanges the variables $(\bs\sigma_i,\bs\lambda_i,\bs r_i)$ with its neighbors, and does
    \begin{align*}
        \bs{\sigma}_i^{k+1} &= \sum_{j \in \mc{N}_i} w_{i,j} \bs{\sigma}_j^k + \phi_i(x_i^{k+1}) -\phi_i(x_i^{k})
        \\
        \bs{r}_i^{k+1} &= \sum_{j \in \mc{N}_i} w_{i,j} \bs{r}_j^k + A_i x_i^{k+1} -A_i x_i^{k} 
        \\
        \bs{\lambda}_i^{k+1} &= \sum_{j \in \mc{N}_i} w_{i,j} \bs{\lambda}_j^k + z_i^{k+1} - z_i^{k}
    \end{align*}
\end{itemize}
\end{algorithm}

Let us define $ \bs \sigma = \col((\bs\sigma_i)_{i\in\mc{I}})$, $ \bs r = \col((\bs r_i)_{i\in\mc{I}})$, $ \bs \lambda = \col((\bs\lambda_i)_{i\in\mc{I}})$ and
the extended game mapping
\begin{align}
    \bs F(x,\bs\sigma)  \coloneqq \col((\bs F_i(x_i,\bs\sigma_i))_{i\in\mc{I}}).
\end{align}
Note that $\bs F(x,\1 \otimes \sigma(x) ) = F(x)$. Furthermore, let
\begin{subequations}
    \label{eq:projected_variables}
\begin{align}
    \bar{\bs{\lambda}}  & = \textstyle \frac{1}{N}\sum_{i\in\mc{I}}\bs\lambda_i, \quad
    \tilde{\bs{\lambda}} = \bs\lambda - \1\otimes \bar{\bs{\lambda}} 
    \\
    \bar{\bs{\sigma}}  & = \textstyle \frac{1}{N}\sum_{i\in\mc{I}}\bs\sigma_i, \quad
    \tilde{\bs{\sigma}} = \bs\sigma - \1\otimes \bar{\bs{\sigma}}  
    \\
     \bar{\bs{ r}}  & = \textstyle \frac{1}{N}\sum_{i\in\mc{I}}\bs r_i, \quad
    \tilde{\bs{ r}} = \bs r - \1\otimes \bar{\bs{r}}.
\end{align}
\end{subequations}

The following lemma shows an invariance property typical of tracking dynamics as those in Algorithm~\ref{algo:1}.
\begin{lemma}\label{lem:invariance}
    For all $k \in \mathbb{N}$, it holds that $
      \bar{\bs{\lambda}}^k =  \frac{1}{N} \textstyle \sum_{i\in\mc{I}} z_i^k$,
      $\bar{\bs{r}}^k =  \frac{1}{N} \textstyle \sum_{i\in\mc{I}} (A_i x_i^k - b_i)$, $
      \bar{\bs{\sigma}}^k =  \sigma(x^k)$.
 \hfill $\square$

 \begin{proof}
     Via induction, by  the initialization and double stochasticity of $W$. 
 \end{proof}   

 To study the convergence of Algorithm~\ref{algo:1}, we first need the following crucial reformulation.
\end{lemma}
\begin{lemma} \label{lem:rewritten_alg}
    The iteration in Algorithm~\ref{algo:1} is equivalent to 
\begin{align}
\label{eq:algo1_compact}
\begin{aligned}
\begin{bmatrix}
    x^{k+1}
    \\
    \bar{\bs \lambda}^{k+1}   
\end{bmatrix}  \!\! = \! \underbrace{\bs{\xi}^k -\alpha \mc{A}(\bs{\xi}^k)}_{\coloneqq\mc{B}_1(\bs{\xi}^k)} 
+\!
\underbrace{\begin{smallbmatrix}
\!\alpha \left( \bs{F} (x^k,\bs{\sigma}^k) -  \bs{F} (x^k, {\bs 1 \otimes}\bar{\bs{\sigma}}^k) + \bs{A}^\top \tilde{\bs\lambda}^k \right) \!
\\
\0
\end{smallbmatrix}}_{\coloneqq\mc{B}_2(\bs \omega^k)}
\\
\begin{bmatrix}
   \tilde{\bs \sigma} ^{k+1}
    \\
    \tilde{\bs r} ^{k+1} 
    \\
    \tilde{\bs \lambda} ^{k+1} 
\end{bmatrix} 
= 
\underbrace{\begin{bmatrix}
        \bs W \tilde {\bs \sigma}^k  
        \\
        \bs W \tilde {\bs r}^k  
        \\
        \bs W \tilde {\bs \lambda}^k  
    \end{bmatrix}}_{ \coloneqq\mc{B}_3(\bs{\chi}^k)}
+ 
\underbrace{
\begin{bmatrix}
         \tilde \Pi \col( (\phi_i(x_i^{k+1})- \phi_i(x_i^{k}) 
        )_{i\in \mc{I}}) 
        \\
        \tilde \Pi \col( (A_i(x_i^{k+1}-x_i^{k})) 
        _{i\in \mc{I}})
        \\
         \alpha N \tilde{\bs r}^k 
    \end{bmatrix}}_{\coloneqq \mc{B}_4(\bs\omega^k)}
    \\[-0.7em]
    \end{aligned}
 \end{align} 
where $\bs\omega = (\bs{\xi},\bs{\chi})$, $\bs{\xi}=(x,\bar{\bs \lambda})$, $\bs{\chi} = (\tilde{\bs \sigma},\tilde{\bs r},\tilde{\bs \lambda})$, and $\tilde \Pi \coloneqq I-(\frac{1}{N} \1\1^\top)\otimes I_N$, $\bs W = W\otimes I$, ${\bs A = \mathrm{diag}((A_i)_{i\in\mc I})}$: the sequence $(x^k,\bs \sigma^k, \bs r^k, \bs\lambda^k)_{k\in\mathbb N}$ generated by Algorithm~\ref{algo:1} and the sequence   $(x^k, \tilde {\bs \sigma}^k+\1\otimes \sigma(x^k) , \tilde{\bs r}^k+\1\otimes (Ax^k-b), \bar{\bs\lambda}^k+\tilde{\bs\lambda}^k)_{k\in\mathbb N}$ generated by \eqref{eq:algo1_compact} coincide. \hfill $\square$
\end{lemma}
\begin{proof} The update of $\bs{\chi}$ follows by noting that $\tilde{\bs \sigma} = \tilde\Pi \bs \sigma$, $\tilde {\bs r} = \tilde\Pi \bs r$, $\tilde {\bs \lambda}  = \tilde\Pi \bs \lambda$ and $ \tilde\Pi \bs W \tilde\Pi = \bs W \tilde\Pi$. 
Then, the proof follows by definition of $\mc{A}$, by using Lemma~\ref{lem:invariance}, and finally by noting that $\bs F (x^k, \1\otimes \bar{\bs \sigma}^k) = F(x^k)$ by Lemma~\ref{lem:invariance}. Note that $\bs{\sigma} = \tilde{\bs \sigma} + \1\otimes \sigma(x) $ and $\bs{r} = \tilde{\bs r} + \frac{1}{N}\otimes (Ax-b)$, which allows us to eliminate the variables $ \bar{\bs{\sigma}}$ and $ \bar{\bs{r}}$ in the iteration, and similarly for $z_i$, since $\bar{\bs \lambda}^{k+1}  = \bar{\bs \lambda}^{k} + \frac{1}{N} \sum_{i\in\mc{I}}(z_i^{k+1}-z_i^{k}) =  \bar{\bs \lambda}^{k} + \frac{1}{N} \sum_{i\in\mc{I}}(\alpha N \bs r_i ^k) = \bar{\bs \lambda}^{k} + \alpha(Ax^k-b)$ by Lemma~\ref{lem:invariance}.
\end{proof}

\begin{theorem}[\gls{GNE} seeking in partial-decision information]\label{th:2}
Let Assumptions \ref{asm:Convexity}-\ref{asm:graph} hold. Then, there exists $\alpha_{\max}>0$ such that, for all $0< \alpha< \alpha_{\max}$, the iteration in \eqref{eq:algo1_compact} converges linearly to $\bs \omega^\star = (\bs{\xi}^\star,\0)$, with $\bs{\xi}^\star =(x^\star,\lambda^\star)$: for all $k \in \mathbb{N}$
\begin{align}
V(\bs \omega^{k+1}) \leq \eta  V(\bs \omega^{k}),
\end{align}
where 
\begin{align}\label{eq:Lyap}
    V(\bs{\omega}) \coloneqq \textstyle\frac{1}{2}\|\bs{\xi}-\bs{\xi}^\star\|_P^2+ \|\bs{\chi}\|^2
\end{align}
for some $\eta<1$. \hfill $\square$
\end{theorem}
\begin{proof}
Note that the operator $\mc{B}_3$ in \eqref{eq:algo1_compact} is a contraction by  Assumption~\ref{asm:graph}, definition of $\bs{\chi}$ and \eqref{eq:projected_variables}; instead $\mc{B}_1$ is a contraction for $\alpha$ small enough as in Theorem~\ref{th:1}. Moreover, the mappings $\mc{B}_2$ and $\mc{B}_3$ are Lipschitz continuous {with constant proportional to the step size $\alpha$,} by the assumptions.  
Therefore, with Lemma~\ref{lem:KKT_strong_mon} and the reformulation in \eqref{eq:algo1_compact} in place, the proof can be
carried out via standard small-gain arguments, and is hence only sketched here. By the Cauchy--Schwarz inequality, we can bound
\begin{align*}
    & \hphantom{{}={}} \| \mc{B}_1(\bs{\xi}^k)
    +\mc{B}_2(\bs{\omega}^k) - \bs{\xi}^\star \|_P
    \\
    & \leq 
    \begin{multlined}[t]
        (1 -{2}\alpha\mu_{\mc{A}}+\alpha^2\ell_{\mc{A}}^2)\|\bs{\xi}^k-\bs{\xi}^\star \|^2_P 
    \\ +
    \alpha \ell_1 \|\bs{\xi}^k-\bs{\xi}^\star \|_P\| \bs{\chi}^k\| 
     +
    \alpha^2 \ell_2 \| \bs{\chi}^k\|^2,
        \end{multlined}
    \end{align*}
    and 
   \begin{align*}
    & \hphantom{{}={}} \| \mc{B}_3(\bs{\chi}^k)
    +\mc{B}_4(\bs{\omega}^k)  \|^2
    \\
    & \leq 
    \begin{multlined}[t]
        \theta \|\bs{\chi}^k\|^2 
    +
    \alpha \ell_3 \|\bs{\xi}^k-\bs{\xi}^\star \|_P\| \bs{\chi}^k\| + \alpha \ell_4\|\bs{\chi}^k\|^2 
    \\
    +
    \alpha^2 \ell_5 \| \bs{\xi}^k -\bs\xi_1^\star \|_P^2
    +\alpha^2 \ell_6 \| \bs{\chi}^k\|^2,
        \end{multlined}
    \end{align*}
where $\ell_1,\ell_2,\ell_3,\ell_4,\ell_5,\ell_6>0$ are parameters independent of $\alpha$. Therefore, for the norm $V$ in \eqref{eq:Lyap} we have  
\begin{align*}
   & \hphantom{{}={}} V(\bs{\omega}^{k+1}) = 
   \begin{smallbmatrix}
       \| \bs{\xi}^{k+1} -\bs\xi^\star\|_P
       \\  
       \|\bs{\chi}^{k+1}\|
   \end{smallbmatrix}^\top
   \begin{smallbmatrix}
       \| \bs{\xi}^{k+1} -\bs\xi^\star\|_P
       \\  \| \bs{\chi}^{k+1}\|
   \end{smallbmatrix}
   \\
   & \leq 
    \begin{smallbmatrix}
       \| \bs{\xi}^{k} -\bs\xi^\star\|_P
       \\  \| \bs{\chi}^{k}\|
   \end{smallbmatrix}^\top 
 M 
    \begin{smallbmatrix}
       \| \bs{\xi}^{k} -\bs\xi^\star\|_P
       \\  \| \bs{\chi}^{k}\|
   \end{smallbmatrix}
   \\& 
   \leq \uplambda_{ {\max}} (M) V(\bs\omega^k),
\end{align*}
where 
\begin{align*}
 M\coloneqq\begin{bmatrix}
        1- {2}\alpha\mu_{\mc{A}}+\alpha^2(\ell_{\mc{A}}+\ell_5)  
        &
        \alpha \frac{1}{2}(\ell_1+\ell_3) 
        \\[1em]
        \alpha \frac{1}{2}(\ell_1+\ell_3) 
        &
        \theta +\alpha\ell_4 +\alpha^2 ({\ell_2+}\ell_6)
    \end{bmatrix}
\end{align*}
  and $\uplambda_{\max}(M) = \|M\|<1$ for small enough $\alpha >0$. 
    \end{proof}

Note that, for the special case $A = \1^\top$, a linearly convergent continuous-time method for \gls{GNE}  seeking in partial-decision information was studied in \cite{Deng_Nian_Aggregative_TNNLS_2019}. Yet, to our knowledge, Theorem~\ref{th:2} is the first result to ensure linear convergence in the case of more general general full-rank constraints. Due to space limitations, we do not derive here an explicit expression for $\alpha_{\max}$ and $\eta$ (or $\mu_{\mc{A}}$ in Lemma~\ref{lem:KKT_strong_mon}). 

\begin{remark}
The proof of Theorem~\ref{th:1} directly applies to the case of a time-varying graph with weight matrix $W^k$, provided that Assumption~\ref{asm:graph} holds for each $k\in\mathbb{N}$.
With some modification, the argument can be extended also to the case of doubly stochastic graphs that are not strongly connected at each step, but such that  $\| W^{kQ} W^{kQ+1} \dots W^{(k+1)Q} - \frac{1}{N} \1\1^\top \| \leq \theta <1$, for a $Q>0$ and all $\k$. 
\hfill $\square$
\end{remark}
    
\section{Equilibrium tracking in time-varying games}
\label{sec:timevarying}

We now consider the case where the game in \eqref{eq:game} varies over time at a rate such that we can not assume a time-scale separation between the game evolution and the GNE seeking iterations. For each time index $t\in\mathbb{N}$, the agents acquire a new instance of the game:
\begin{align} \label{eq:game_tvar}
	(\forall i \in \mathcal{I})
	\  \min_{ y_i \in \R^{ {n_{i}}}}   \; J^t_i(y_i,x_{-i}) \quad
	\text{s.t.}  \  (y_i,x_{-i}) \in \mathcal X^t.
	\end{align}
 We consider the case when the constraints of the game vary only in their affine part, that is, $\mc X^t:=\{x\in\R^n|Ax=b^t\}$, and Assumptions \ref{asm:Convexity}--\ref{asm:constraintrank} hold for each $t$. The games in \eqref{eq:game_tvar} define a primal-dual GNE pair \emph{sequence} $(\omega^{\star}_t)_{t\in\mathbb{N}}$, corresponding to the zero set of the KKT operators $(\mc A_t)_{t\in\N}$ defined for all $t$ as in \eqref{eq:KKT}, with $F, b$ replaced respectively by $F^t:=\col((\nabla_{x_i}J_i^t)_{i\in\mc I})$ and $b^t$. The GNE sequence is unique for each $t$ following the strong monotonicity of $\mc A_t$ and \cite[Ex. 22.12]{Bauschke_2017}.
 As the rate at which the problem varies is comparable to the agents' computation time, the agents can only compute an approximation of the GNE at time $t$ before they are presented with a new instance of the problem. The goal of the agents is then to find a sequence $(\omega^t)_{t\in\mathbb{N}}$ which asymptotically tracks relatively well the GNE sequence.
    We formulate the following assumption, which is standard in the literature of online optimization \cite[Assm. 1]{Simonetto_time_IEEEProc_2020}, \cite[Eq. 9]{DallAnese_Optimization_SPM_2020} and is verified, for example, for games affected by a bounded process noise in the linear constraints \cite[Lemma 5]{su_Online_AIS_2021}.
 \begin{assumption} \label{asm:bounded_t_var} For some $\delta \geq 0$, it holds that the solution $\omega_t^\star$ of the game in \eqref{eq:game_tvar} satisfies
     $$\sup_{t\in\mathbb N} \| \omega^{\star}_{t+1} - \omega^{\star}_t \| \leq \delta. $$
 \end{assumption}
Assumption \ref{asm:bounded_t_var} implies that the solution at time $t$ is an approximate solution for the problem at time $t+1$. Given an estimate of the solution at time $\omega^{t-1}$ for some time step $t$, we then propose to compute $\omega^{t}$ by performing $K$ iterations of the iteration in \eqref{eq:algorithm_full},
warm-started at $\omega^{t}$, that is: 
\begin{subequations} \label{eq:t_var_algorithm}
    \begin{align}
        y^{t,0} &= \omega^{t-1}
        \\
        y^{t,k+1} &=  
         y^{t,k} - \alpha \mc A_t(y^{t,k}) & \text{for}~k=0,...,K-1
        \\
        \omega^{t} &= y^{K},
    \end{align}
\end{subequations}
where $(y^{t,k})_{k\in\{1,...,K\}}$ are auxiliary variables. The following lemma shows that, for an appropriately chosen step size, the proposed algorithm tracks the GNE trajectory up to an asymptotic error which depends on $\delta$ and $K$.
\begin{theorem}
For any $0<\alpha<\frac{2\mu_{\mc A}}{\ell_{\mc A}^2}$, $\omega^0$, $K \in \mathbb{N}_{>0}$, the sequence $(\omega^t)_{t\in\N}$ generated by the iteration in \eqref{eq:t_var_algorithm} satisfies
\begin{equation}
    \limsup_{t\rightarrow\infty} \| \omega^t - \omega^{\star}_t \|_P \leq  \frac{\rho^{K/2}}{1-\rho^{K/2}}\delta \sqrt{\uplambda_{\text{max}}(P)}.
\end{equation}
where $\rho$ is as in Theorem \ref{th:1}. \hfill $\square$ \label{th:3}
\end{theorem}
\begin{proof}
Following Theorem \ref{th:1}, for $\alpha\leq 2\frac{\mu_{\mc A}}{\ell_{\mc A}^2}$,
\begin{equation*}
    \|\omega^{t} - \omega_{t}^{\star} \|_P \leq \rho^{K/2}\|\omega^{t-1} - \omega_{t}^{\star}\|_P.
\end{equation*}
From the latter, the triangle inequality and the fact $\|z\|^2_P \leq \uplambda_{\text{max}}(P)\|z\|^2$ for all $z$:
\begin{align*}
    \|\omega^{t} - \omega_{t}^{\star} \|_P &\leq \rho^{K/2}\|\omega^{t-1} - \omega_{t-1}^{\star}\|_P + \rho^{K/2}\| \omega_t^{\star} - \omega_{t-1}^{\star} \|_P\\
    & \leq \rho^{K/2}\|\omega^{t-1} - \omega_{t-1}^{\star}\|_P + \rho^{K/2} \delta\sqrt{\uplambda_{\text{max}}(P)} .
\end{align*}
Iterating the latter $t$ times, we obtain
\begin{align*}
    \|\omega^{t} - \omega_{t}^{\star} \|_P &\leq \rho^{Kt/2}\|\omega^0 - \omega_0^{\star}\|_P + \sum_{\tau=1}^t \rho^{K\tau/2}\delta\sqrt{\uplambda_{\text{max}}(P)}.
\end{align*}
Since $\rho^{K/2}<1$, the thesis follows by the convergence of the geometric sequence.
\end{proof}
Let us now turn our attention to the time-varying counterpart of the partial-decision information setup described in Section \ref{sec:partial_info}. Again, we consider aggregative games in the form 
\begin{equation*}
    J^t_i(x_i,x_{-i}) = J^t_i(x_i,\sigma^t(x))
\end{equation*}
where $\sigma^t(x):=\frac{1}{N}\sum_{i\in\mc I} \phi_i^t(x_i)$ and we postulate that $\phi_i^t$ satisfies Assumption \ref{asm:aggregation} for all $t$. As in Section \ref{sec:partial_info}, we augment the state of each agent with an estimate of $\sigma^t(x)$, of the dual variable and of the residual $Ax-b^t$. For every $t$, denote $\bs\omega^{\star}_t=(\omega_t^\star, \bs 0)$, where $\omega_t^\star=(x^{\star}_t, \lambda^{\star}_t)$ is a primal-dual solution of the game at time $t$ and the vector of zeros represents the target estimation error. We then define the reference trajectory as $(\bs\omega_t^{\star})_{t\in\N}$. At each time-step, we propose to appropriately re-initialize the dynamic tracking of the estimated variables and, in the spirit of the iteration in \eqref{eq:t_var_algorithm}, to apply a finite number of iterations of Algorithm \ref{algo:1}. The resulting method is illustrated in Algorithm \ref{algo:2}. We obtain the following counterpart of Lemma \ref{lem:invariance} for the re-initialized dynamic tracking.

\begin{algorithm}[t] \caption{Time-varying fully-distributed GNE seeking} \label{algo:2}
\emph{Initialization}: choose $\alpha>0$ as in Theorem~\ref{th:2}; for all $i\in \mc{I}$, set $x_i^0 \in \R^{n_i}$, $\bs\sigma_i^0 = \bs 0$, $ z_i^0 \in \R^m$, $\bs{\lambda}_i ^0 = z_i^0$, $\bs r_i^0 = A_ix_i^0$, $b_i^{0}=\bs 0$, $\phi_i^{0}(\cdot)=\bs 0$.
\newline
\emph{Iteration: at time $t\in\N_{>0}$, for each agent $i\in\mc I$},
\begin{enumerate}
\item Acquire $J_i^t$, $\phi_i^t$, $b_i^t$
\item Re-initialize $\hat{x}_i^{t,0} = x_i^{t-1}$, $\hat{z}_i^{t,0} = z_i^{t-1}$, $\hat{\bs\sigma}_i^{t,0} = {\bs\sigma}_i^{t-1}-\phi^{t-1}_i(x_i^{t-1})+\phi^t_i(x_i^{t-1})$, $\hat{\bs r}_i^{t,0} = \bs{r}_i^{t-1} + b_i^t - b_i^{t-1}$, $\hat{\bs\lambda}_i^{t,0} = {\bs\lambda}_i^{t-1}$ 
    \item For all $k \in \{0,...,K-1\}$, for all $i\in \mc{I}:$
\begin{itemize}
    \item  Local variables update: 
    \begin{align*}
        \hat{x}_i^{t,k+1} & = \hat{x}_i^{t,k} - \alpha \left ( \bs{F}^t_i(\hat{x}_i^{t,k}, \hat{\bs{\sigma}}_i^{t,k}) + A_i^\top \hat{\bs{\lambda}}_i^{t,k} \right) 
        \\
        \hat{z}_i^{t,k+1} & = \hat{z}_i^{t,k} + \alpha N \hat{\bs r}_i^{t,k}
    \end{align*}
    \item Estimation update: Agent $i$ exchanges the variables $(\hat{\bs\sigma}_i^{t,k},\hat{\bs\lambda}^{t,k}_i,\hat{\bs r}^{t,k}_i)$ with its neighbors, and updates
    \begin{align*}
        \hat{\bs{\sigma}}_i^{t,k+1} &= \sum_{j \in \mc{N}_i} w_{i,j} \hat{\bs{\sigma}}_j^{t,k} + \phi^t_i(\hat{x}_i^{t,k+1}) -\phi^t_i(\hat{x}_i^{t,k})
        \\
        \hat{\bs{r}}_i^{t,k+1} &= \sum_{j \in \mc{N}_i} w_{i,j} \hat{\bs{r}}_j^{t,k} + A_i \hat{x}_i^{t,k+1} -A_i \hat{x}_i^{t,k} 
        \\
        \hat{\bs{\lambda}}_i^{t,k+1} &= \sum_{j \in \mc{N}_i} w_{i,j} \hat{\bs{\lambda}}_j^{t,k} + \hat{z}_i^{t,k+1} - \hat{z}_i^{t,k}
    \end{align*}
\end{itemize}
\item Set $x_i^{t}= \hat{x}_i^{t,K}$, $z_i^{t}=\hat{z}_i^{t,K}$, $\bs\sigma_i^{t}=\hat{\bs\sigma}_i^{t,K}$, $\bs r_i^{t} = \hat{\bs r}_i^{t,K}$, $\bs \lambda_i^{t} = \hat{\bs\lambda}_i^{t,K}$
\end{enumerate}

\end{algorithm}

\begin{lemma}\label{lem:invariance_2}
    For all $t\in\N$, $\bar{\bs\lambda}^{t}=\sum_{i\in\mc I} ({z}_i^{t})$, it holds that  $\bar{\bs r}^{t}=\sum_{i\in\mc I} (A_i{x}_i^{t}-b_i^t)$, $\bar{\bs\sigma}^{t}=\sigma^t({x}^{t})$. \hfill$\square$
\end{lemma}
\begin{proof}
    Let for some $t$:
    \begin{align}\label{eq:lem:invariance_2:1}
        \bar{\bs r}^{t} =\tfrac{1}{N} \textstyle\sum_{i\in\mc I} A_i{x}_i^{t}-b_i^t.
    \end{align}
    From the update step 4) and the re-initialization step 2), we obtain ${x}_i^{t}=\hat{x}_i^{t,K}=\hat{x}_i^{t+1,0}$ and
    \begin{align*}
        \tfrac{1}{N}\textstyle\sum_{i\in\mc I}\hat{\bs r}_i^{t+1,0} &= \bar{\bs r}^{t} - \tfrac{1}{N}\textstyle\sum_{i\in\mc I} (b_i^{t+1}-b_i^{t})\\
        &\overset{\eqref{eq:lem:invariance_2:1}}{=}\tfrac{1}{N} \textstyle\sum_{i\in\mc I} (A_i\hat{x}_i^{t+1,0}-b_i^{t+1}).
    \end{align*}
    Then, by Lemma \ref{lem:invariance}, $\frac{1}{N}\sum_{i\in\mc I}\hat{\bs r}_i^{t+1,K} =\tfrac{1}{N} \textstyle\sum_{i\in\mc I} A_i\hat{x}_i^{t+1,K}-b_i^{t+1}$. From the latter and the update step 4), 
    $$\bar{\bs r}^{t+1}=\tfrac{1}{N}\textstyle\sum_{i\in\mc I}\hat{\bs r}^{t+1,K}=\tfrac{1}{N} \textstyle\sum_{i\in\mc I} A_i{x}_i^{t+1}-b_i^{t+1} . $$ 
    The result then follows by induction and similarly for $\bar{\bs\sigma}^{t}$ and $\bar{\bs\lambda}^{t}$.
\end{proof}
The re-inizialization of the dynamic tracking introduces an additional error term in the solution tracking, which requires the following technical assumption on the time variation of the functions $(\phi_i^t)_{t\in\N}$:
\begin{assumption} \label{as:delta_phi} For some $\delta_{\phi}>0$, it holds that
    \begin{equation*}
        \sup_{x,t}\|\col(\phi_i^t(x_i)-\phi_i^{t+1}(x_i))_{i\in\mc I}\| \leq \delta_{\phi}.
    \end{equation*}\hfill$\square$
\end{assumption}

\begin{theorem} \label{th:4} Let ${\bs\omega}^t=\col(x^t, \bar{\bs \lambda}^{t}, \tilde{\bs\sigma}^{t}, \tilde{\bs r}^{t} , \tilde{\bs\lambda}^{t} ).$
    There exists $\alpha_{\text{max}}$ such that, for every $0<\alpha<\alpha_{\text{max}}$, the sequence $(\bs\omega^t)_{t\in\N}$ generated by Algorithm \ref{algo:2} satisfies 
\begin{align}
\begin{split}
    \limsup_{t\rightarrow\infty} &\| \bs\omega^t - \bs\omega^{\star}_t \|_Q  \leq \\
    &\tfrac{\eta^{K/2}}{1- \eta^{K/2}}\sqrt{\uplambda_{\text{max}}(Q)}((1+N\ell_A)\delta + \delta_{\phi})
\end{split}
\end{align}
where $Q=\mathrm{diag}(P/2,I)$, for some $\eta\in(0,1)$.\hfill$\square$
\end{theorem}
\begin{proof} Following the same steps as in Lemma \ref{lem:rewritten_alg}, the inner iteration (Step 3) of Algorithm \ref{algo:2} is equivalent to the iteration in \eqref{eq:algo1_compact}, where $\bs F$ and $\phi_i$ are substituted by their time-varying counterpart. Now denote 
\begin{align*}\psi^t =&\col(\bs 0, \bs 0, \col((\phi_i^t(x_i^t)-\phi_i^{t-1}(x_i^t))_{i\in\mc I}, \\
&\col(b_i^t-b_i^{t-1})_{i\in\mc I}, \bs 0).
\end{align*}
From Assumptions \ref{asm:bounded_t_var}, \ref{as:delta_phi} and from $A x_t^{\star} = b^t$,
\begin{align}\label{eq:bound_psi}
\begin{split}
    \|\psi^t\| &\leq \delta_{\phi} + \textstyle\sum_{i\in\mc I} \| b_i^t - b_i^{t-1} \| \\
    &\leq \delta_{\phi} + N\|b^t-b^{t-1} \|\leq  \delta_{\phi} + N \ell_A \delta 
\end{split}
\end{align}
From Theorem \ref{th:2} and accounting for the re-initialization step, we find for every $t$:
    \begin{align} \label{eq:t_var_contraction_partinfo}
        & \|  \bs\omega^t - \bs\omega^{\star}_t\|_Q \leq \eta^{K/2} \| {\bs\omega}^{t-1} + 
        \psi^t - \bs\omega^{\star}_t\|_Q
    \end{align}
    for some $\eta\in(0,1)$. By the triangle inequality, Assumption \ref{asm:bounded_t_var} and \eqref{eq:bound_psi}, and from the fact $\|z\|^2_Q \leq \uplambda_{\text{max}}(Q)\|z\|^2$, we have
    \begin{align*}
        &\| {\bs\omega}_{t-1} + 
        \psi_t - \bs\omega^{\star}_t\|_Q \leq \| {\bs\omega}_{t-1} - \bs\omega^{\star}_t \|_Q + \| \psi_t \|_Q \\ 
        &\leq \| {\bs\omega}_{t-1} - \bs\omega^{\star}_{t-1} \|_Q + \| \bs\omega^{\star}_{t} -\bs\omega^{\star}_{t-1}\|_Q + \| \psi_t \|_Q\\
        &\leq \| {\bs\omega}_{t-1} - \bs\omega^{\star}_{t-1} \|_Q + \sqrt{\uplambda_{\text{max}}(Q)}((1+N\ell_A)\delta + \delta_{\phi}).
    \end{align*}
    By substituting the latter in \eqref{eq:t_var_contraction_partinfo} and by iterating the resulting inequality, we obtain: 
    \begin{align*}
        \|  \bs\omega_t - \bs\omega^{\star}_t\|_Q & \leq  \eta^{Kt/2} \| {\bs\omega}_{0} - \bs\omega^{\star}_0\|_Q \\
        & + \textstyle\sum_{\tau=1}^t \eta^{Kt/2}\sqrt{\uplambda_{\text{max}}(Q)}((1+N\ell_A)\delta + \delta_{\phi}).
    \end{align*}
    Then, as $\eta^{K/2}<1$, the result follows from the convergence of the geometric series.
\end{proof}

\begin{remark}
   Considering a time-varying matrix constraint $A_t$ (instead of $A$) would generate some complications, as also the matrix wou \eqref{eq:def_P} would be time-varying. This case can be dealt by assuming a lower bound for $K$ in Theorems~\ref{th:3} and \ref{th:4}, or under the extra assumption that $\bar A A_t^\top \geq \mu_A$ for a matrix $\bar A$ and all $\k$; but it is not discussed here. \hfill $\square$
\end{remark}

\section{Numerical example$^1$}
\label{sec:num}
\let\thefootnote\relax\footnote{$^1$\texttt{https:\slash{\slash}github.com{\slash}bemilio{\slash}Simple\_peer\_to\_peer} }

We demonstrate the proposed algorithms on a market clearing problem for a peer-to-peer energy market model inspired by the one in \cite{Belgioioso_Energy_ECC_2020}. We consider $N=6$ prosumers that aim at determining their energy portfolio. At each time-step $t$, the agents can either purchase power from a main energy operator, produce it from a dispatchable energy source or trade it with their respective neighbors over a randomly generated undirected graph $\mc G^{\textrm{E}}$. Furthermore, the agents can exchange information over an undirected connected communication graph $\mc{G}$. We denote for each agent $i$ and each time-step $t$ the power purchased from the main operator as $x^{\text{mg}}_{i,t}$, the produced power as $x^{\text{dg}}_{i,t}$ and the power that agent $i$ purchases from agent $j$ as $x^{\text{tr}}_{i,j,t}$, $j\in\mc N_i^{\textrm E}$, with $\mc N_i^{\textrm E}$ the set of neighbors of agent $i$ over $\mc{G}^{\textrm{E}}$. As in \cite{Belgioioso_Energy_ECC_2020}, the energy price posed by the main operator increases linearly with the aggregate power requested at the main energy operator by a factor $c_{\text{mg}}>0$. 
Thus, by defining the aggregative value
$$\sigma^{\text{mg}}(x^\text{mg}) = \textstyle\sum_{i\in\mc I} x^{\text{mg}}_{i},$$
the cost incurred by each agent for purchasing energy from the operator is $J^{\text{mg}}(x^{\text{mg}}_{i,t}, \sigma^{\text{mg}}(x^\text{mg}_t))= c^{\text{mg}}\sigma^{\text{mg}} (x^\text{mg}_t) x^{\text{mg}}_{i,t}$.
We consider quadratic cost on the power generation incurred by the agents \cite[Eq. 2]{Belgioioso_Energy_ECC_2020}, with the form $J^{\text{dg}}_{i,t}(x^{\text{dg}}_{i,t})= c^{\text{dg}}(x^{\text{dg}}_{i,t} - x^{\text{dg, ref}}_{i,t})^2$ where $x^{\text{dg, ref}}_{i,t}$ is the  \emph{time-varying} scheduled setpoint of the dispatchable generators and $c^{\text{dg}}>0$. The price (or revenue) of trading energy between peers is linear \cite[Eq. 10]{Belgioioso_Energy_ECC_2020}, and we assume that the agents incur a quadratic cost on the transactions for utilizing the market, thus the total objective function related to the peer-to-peer trading is given by $J^{\text{tr}}((x^{\text{tr}}_{i,j,t})_{j\in\mc N_i^{\textrm E}} ) = \sum_{j\in\mc N_i} (c^{\text{tr}} x^{\text{tr}}_{i,j,t} + \kappa^{\text{tr}} (x^{\text{tr}}_{i,j,t})^2)$. We impose that the agents cannot sell power to the main operator and that, due to physical limitations, the power generated by the dispatchable units must be non-negative. As the formulation in \eqref{eq:game} does not consider inequality constraints, this is enforced by a Lipschitz continuous approximation of the logarithmic barrier function 
$$J^{\text{bar}}(x^{\text{mg}}_{i,t}, x^{\text{dg}}_{i,t}) = \Gamma_\gamma(x^{\text{mg}}_{i,t}) + \Gamma_\gamma(x^{\text{dg}}_{i,t}),$$
where $\Gamma_\gamma(y):= \min(-\log(y), -\gamma y + 1-\log(1/\gamma) )$ for $\gamma\gg0$. The total cost incurred by each agent is thus given by 
\begin{align*}
    J_i^t(x_{i,t},x_{-i,t}) &=  J^{\text{mg}}(x^{\text{mg}}_{i,t}, \sigma^{\text{mg}}(x_t^{\text{mg}}))+J^{\text{dg}}_{i,t}(x^{\text{dg}}_{i,t}) \\ &  +  J^{\text{tr}}((x^{\text{tr}}_{i,j,t})_{j\in\mc N_i^ \textrm{E}}) + J^{\text{bar}}(x^{\text{mg}}_{i,t}, x^{\text{dg}}_{i,t}) .
\end{align*}
Moreover, given a power demand $p_{i,t}^{\text{d}}$, the agents need to satisfy the power balance equation \cite[Eq. 1]{Belgioioso_Energy_ECC_2020}
\begin{equation}\label{eq:constraint_powerbalance}
    \sum_{j\in\mc N_i^{\textrm E}}(x^{\text{tr}}_{i,j,t}) + x^{\text{mg}}_{i,t} + x^{\text{dg}}_{i,t} = p_{i,t}^{\text{d}}.
\end{equation}
\begin{figure}
    \centering
    \includegraphics[width=\columnwidth]{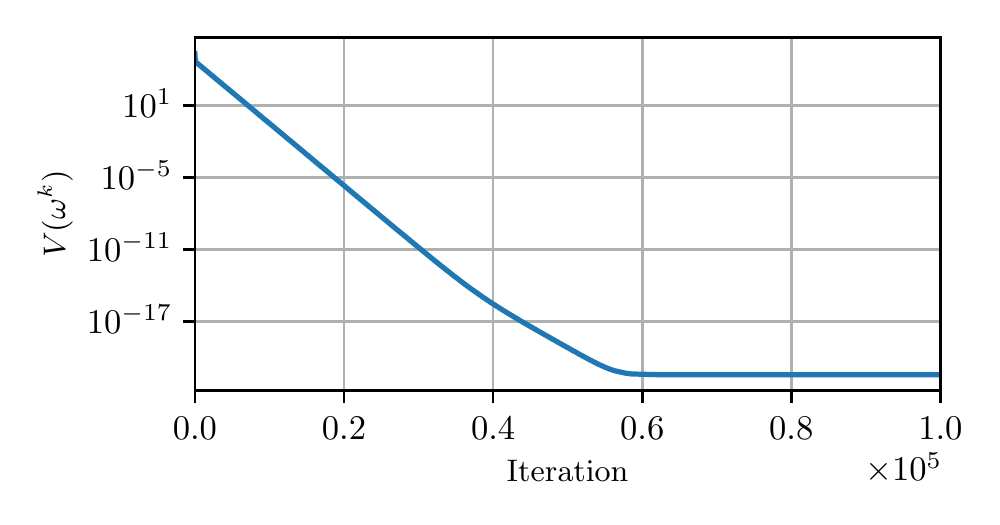}
    \caption{Convergence of Algorithm \ref{algo:1} for the day-ahead market clearing problem.}
    \label{fig:1}
\end{figure}
\begin{figure}
    \centering
    \includegraphics[width=\columnwidth]{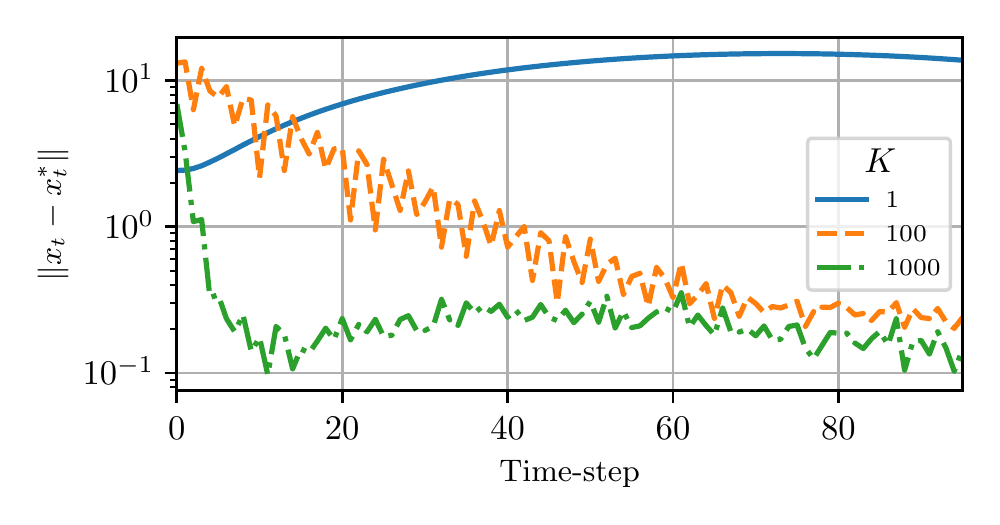}
    \caption{Tracking error of Algorithm \ref{algo:2} for the real-time market clearing problem with respect to the day-ahead solution computed by $10^5$ iterations of Algorithm \ref{algo:1}. } 
    \label{fig:2}
\end{figure}
\begin{figure}
    \centering
    \includegraphics[width=\columnwidth]{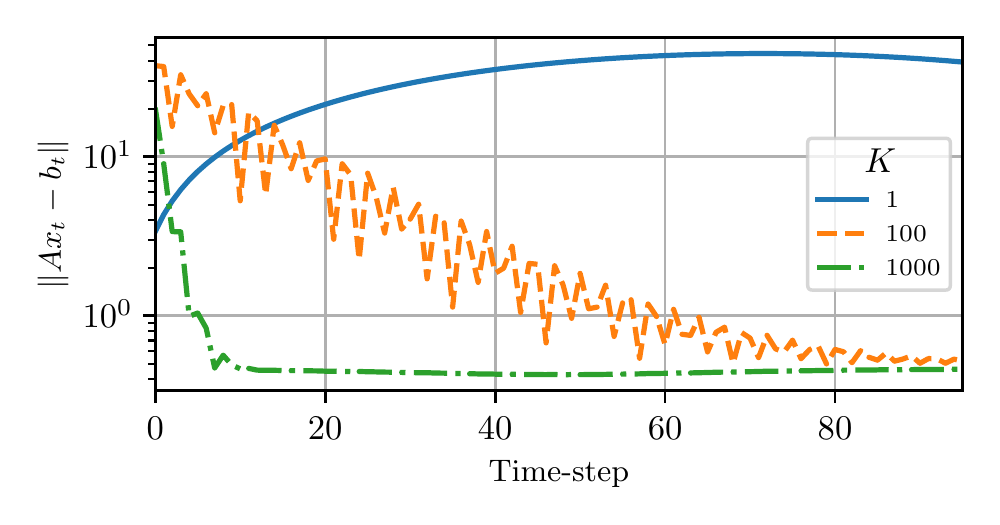}
    \caption{Constraint violation incurred by Algorithm \ref{algo:2} for the real-time market clearing problem. The  inequality constraints enforced through the barrier function $J^{\text{bar}}$ are observed to be always satisfied. }
    \label{fig:3}
\end{figure}
As the power balance constraints are local,  we do not apply the dynamic tracking method to the associated dual variables (i.e., dual variables are managed locally, see  \cite[Rem.~2]{Bianchi_ct_GNE_AUT_2021}). Instead, coupling constraints between the agents decisions arises via  trading reciprocity constraints \cite[Eq. 8]{Belgioioso_Energy_ECC_2020}:
\begin{equation}\label{eq:constraint_reciprocity}
    x_{i,j,t}^{\text{tr}} + x_{j,i,t}^{\text{tr}} = 0.
\end{equation}

We first consider a time-invariant scenario and  compute the day-ahead market clearing solution over an entire day, with  time-steps of 15 minutes: namely,  the cost of agent $i$ is given by $\sum_{t=1}^T J_{i}^t$, and the constraints in \eqref{eq:constraint_powerbalance}-\eqref{eq:constraint_reciprocity} are imposed for all $t = 1,2,\dots,T$, with $T =96$. 
 Figure \ref{fig:1} shows that, as expected,  Algorithm~\ref{algo:1} exhibits a linear convergence rate with respect to the Lyapunov function in \eqref{eq:Lyap}. 

Then, we consider a real-time  scenario. In particular,  the agents only have access to a prediction on their load demand and generation setpoint over the coming quarter of an hour; hence the cost of agent $i$ at each time $t = 1,2,\dots,T$ is given by $J_{i}^t$.  Note that the agents are in fact faced with a time-varying generalized game as discussed in Section~\ref{sec:timevarying}, which we address via Algorithm~\ref{algo:2}. 
The results are shown in Figure~\ref{fig:2}. 
Because of the slow convergence (i.e., $\eta$ is close to $1$), $K=1$ results in a significant tracking error; however, good performance is observed already for $K = 100$. Finally, in Figure \ref{fig:3}, we show the constraint violation obtained by the proposed method over the simulation horizon. As constraints are only satisfied asymptotically, performing only a finite number of iterations per time-step leads to a constraint violation, which as expected decreases with $K$.

\section{Conclusion}
Strongly monotone \gls{GNE} problems with full row rank  equality coupling constraints can be solved with linear convergence rate, both in semi-decentralized and fully-distributed settings, via primal-dual algorithms. The contractivity properties of the iterates also allow the  tracking of the solution sequence in time-varying games; in this online setting, the asymptotic tracking accuracy can be increased by increasing the update frequency.  

As our results exploit the strong monotonicity of the \gls{KKT} operator in a (non-diagonally) weighted space, it is not clear how to embed projections in the proposed methods, which is the main drawback of our approach. Future work should hence focus on linear convergence in generalized games with local constraints (and with inequality coupling constraints).

 
\bibliographystyle{IEEEtran}
\bibliography{library}

\end{document}